\documentclass{amsart}

\usepackage{a4wide,amsmath,amssymb}
\usepackage[toc,page]{appendix}

\usepackage{url}

\usepackage{hyperref}

\usepackage{mathtools}

\usepackage{xcolor}

\newtheorem{thm}{Theorem}[section]

\newtheorem{prop}[thm]{Proposition}

\theoremstyle{remark}
\newtheorem{rem}[thm]{Remark}

\theoremstyle{definition}

\newtheoremstyle{Claim}{}{}{\itshape}{}{\itshape\bfseries}{:}{ }{#1}
\theoremstyle{Claim}

\newcommand{\R}{\mathbb{R}}

\newcommand{\eps}{\varepsilon}

\theoremstyle{plain}

\def\sideremark#1{\ifvmode\leavevmode\fi\vadjust{
\vbox to0pt{\hbox to 0pt{\hskip\hsize\hskip1em
\vbox{\hsize3cm\tiny\raggedright\pretolerance10000
\noindent #1\hfill}\hss}\vbox to8pt{\vfil}\vss}}}
\makeatletter
\@namedef{subjclassname@2020}{\textup{2020} Mathematics Subject Classification}
\makeatother
\begin{document}

\title[]{Interpolated time-H\"older regularity of solutions of fully nonlinear parabolic equations}

\author{Alessandro Goffi}
\address{Dipartimento di Matematica e Informatica ``Ulisse Dini'', Universit\`a degli Studi di Firenze, 
viale G. Morgagni 67/A, 50134 Firenze (Italy)}
\curraddr{}
\email{alessandro.goffi@unifi.it}


 \thanks{
The author is member of the Gruppo Nazionale per l'Analisi Matematica, la Probabilit\`a e le loro Applicazioni (GNAMPA) of the Istituto Nazionale di Alta Matematica (INdAM). He was partially supported by the INdAM-GNAMPA projects 2024 and 2025, the King Abdullah University of Science and Technology (KAUST) project CRG2021-4674 ``Mean-Field Games: models, theory and computational aspects" and by the project funded by the EuropeanUnion – NextGenerationEU under the National Recovery and Resilience Plan (NRRP), Mission 4 Component 2 Investment 1.1 - Call PRIN 2022 No. 104 of February 2, 2022 of Italian Ministry of University and Research; Project 2022W58BJ5 (subject area: PE - Physical Sciences and Engineering) ``PDEs and optimal control methods in mean field games, population dynamics and multi-agent models" 
 }

\date{\today}

\begin{abstract}
We show interior Schauder estimates for a special class of fully nonlinear parabolic Isaacs equations by the maximum principle, providing an Evans-Krylov result for the model equation $\min\{\inf_{\beta}L_\beta u,\sup_\gamma L_\gamma u\}-\partial_t u=0$, where $L_\beta,L_\gamma$ are linear operators with possibly variable H\"older coefficients. We also give a proof of the Evans-Krylov theorem for fully nonlinear uniformly parabolic equations for which a regularity theory of the stationary non-homogeneous equation is available.
\end{abstract}

\maketitle

\section{Introduction}

In this note we provide a maximum principle approach to study interior maximal $C^\alpha$ smoothness (i.e. local $C^{2,\alpha}$ regularity with respect to the parabolic distance) of solutions to the following fully nonlinear parabolic model
\begin{equation}\label{parintro}
F(D^2u)-\partial_t u=0\text{ in }Q_1:=B_1\times(-1,0]\subseteq\R^{n+1},
\end{equation}
where we assume
\begin{equation}\label{Fccvarintro}
\begin{cases}
F(M)=\min\{F^\cap (M),F^\cup(M)\},\ \forall M\in\mathcal{S}_n,\\
F(0)=0,\ F^\cap,F^\cup\text{ are uniformly elliptic},\\
F^\cap\text{ is concave in $M$},\ F^\cup\text{ is convex in $M$}.
\end{cases}
\end{equation}
The uniform ellipticity for an operator $G=G(M):\mathrm{Sym}_n\to\R$, $\mathrm{Sym}_n$ being the space of $n\times n$ symmetric matrices, reads as 
\[
\lambda\|N\|\leq G(M+N)-G(M)\leq \Lambda\|N\|,\ \forall M\in\mathrm{Sym}_n,\ N\geq0,\ \lambda\leq\Lambda,\ \|N\|=\sup_{|x|=1}|Nx|.
\]
Note that the uniform ellipticity of $F^\cap,F^\cup$ implies that $F$ itself is uniformly elliptic. In this case we will call \eqref{parintro} uniformly parabolic. For instance, the result applies to the evolution PDE
\begin{equation}\label{IsaacsIntro}
\min\left\{\inf_{a\in\mathcal{A}}L_au,\sup_{b\in\mathcal{B}}L_bu\right\}-\partial_tu=0,
\end{equation}
where $\mathcal{A},\mathcal{B}$ are arbitrary sets and $L_a,L_b$ are linear operators of the form $L_ku=a_{ij}^k\partial_{ij}u+c_k$, $c_k=L_k0\in\R$, $a_{ij}^k$ with constant coefficients. This class can be written in the Isaacs form $\inf_{\sigma\in\Sigma}\sup_{\theta\in\Theta}L_{\sigma\theta}u-\partial_t u=0$, where the convex operators $\left(\sup_{\theta\in\Theta}L_{\sigma\theta}u\right)_{\sigma\in\Sigma}$ are all linear except for at most one $\sigma$. A prototype equation is the flow driven by the 3-operator
\[
F_3(D^2u)-\partial_tu=\min\{L_1u,\max\{L_2u,L_3u\}\}-\partial_tu,
\]
where $L_1,L_2,L_3$ are linear uniformly elliptic operators with constant coefficients. Our result provides a time-dependent counterpart of the maximal $C^\alpha$ regularity found by X. Cabr\'e and L. Caffarelli \cite{CCjmpa} and it can be regarded as an Evans-Krylov type result for a (time-dependent) nonconvex/nonconcave equation. We recall that the Evans-Krylov property means the boost of regularity
\[
C^{2,1}\rightarrow C^{2+\alpha,1+\alpha/2},
\]
see Theorem 6.1 of \cite{CC} for the elliptic result.\\
The outline of the proof of the maximal parabolic $C^\alpha$ regularity under the assumption \eqref{Fccvarintro} is the following: one first proves that for the homogeneous equation $F(D^2u)-\partial_t u=0$ we have
\[
\partial_t u,Du\in C^{\alpha,\alpha/2}
\]
for some small universal $\alpha\in(0,1)$, by the Krylov-Safonov theory \cite{KrylovSafonov} (here $C_t^{\alpha/2}$ denotes the H\"older continuity in time, uniformly in space). Indeed, the first derivative (or its incremental quotient), both in time and space, satisfies a linear nondivergence form equation, since the equation is invariant in space-time. This step does not require neither convexity nor concavity assumptions on $F$, and it holds for  continuous viscosity solutions of any uniformly parabolic equation, see Theorem \ref{firstorder}. Then, for each time slice, which we freeze and do not display, we write
\[
F(D^2u)=g
\]
where $g=\partial_t u\in C^{\alpha,\alpha/2}$. This implies by the stationary (space) regularity of \cite{CCjmpa}
\[
u\in C^{2,\alpha}_x
\]
at each time layer. To reach parabolic $C^{2+\alpha,1+\alpha/2}$ regularity of the solution to the parabolic equation  (and in particular $\alpha/2$ time-regularity of second derivatives) we exploit a maximum principle method. More precisely, exploiting the information $D^2u\in C^\alpha_x$, we improve the time-regularity of the gradient showing that
\[
Du\in C_t^{\frac{1+\alpha}{2}}
\]
via the comparison principle. We then exploit an interpolation argument between $Du\in C_t^{(1+\alpha)/2}$ and $D^2u\in C^{\alpha}_x$ to conclude that $D^2u\in C^{\alpha/2}_t$, proving thus the claimed $C^{2+\alpha,1+\alpha/2}$ regularity. Note that a direct interpolation with the sole information $\partial_t u,Du\in C^{\alpha,\alpha/2}$ available from the Krylov-Safonov theory fails to yield the expected $\alpha/2$-H\"older regularity in time of $D^2u$. However, it is  enough in general to reach classical smoothness for a smaller exponent $\bar \alpha=\frac{\alpha}{2+\alpha}<\frac{\alpha}{2}$.\\
 In the case of the heat equation, a maximum principle approach was used to prove the time-H\"older regularity of the Hessian in \cite{Knerr} via an argument of A. Brandt \cite{Brandt}. \\
Notably, this method provides a new proof of the parabolic $C^{2+\alpha,1+\alpha/2}$ regularity for fully nonlinear uniformly parabolic equations knowing the stationary $C^{2,\alpha}$ estimate for nonhomogeneous equations via the maximum principle; see the end of Section \ref{sec;Isaacs}. This is the case of fully nonlinear uniformly parabolic equations in two space dimensions \cite{Andrews,SilvestreAnnali}, operators with Cordes conditions \cite{Gjmpa,Vitor}, concave \cite{Evans82,KrylovBookNew} and quasiconcave equations \cite{Gjmpa}, operators of twisted type \cite{Collins16}. A comparison with other approaches to study the parabolic regularity of fully nonlinear equations is in Remark \ref{compare}. \\

The study of these estimates are motivated by the challenging open question of finding some special structures of $F$ (intermediate among the sole uniform ellipticity and weaker than convexity) guaranteeing classical smoothness of solutions when the parabolic dimension $n+1\geq 4$, cf. the introduction of \cite{Gjmpa} and Section \ref{sec;survey} for references on the subject. \\
Following \cite{CCjmpa}, we mention that these interior estimates could be used to prove an existence result for the Dirichlet problem by a refinement of the method of continuity based on weighted interior bounds. This approach was first introduced by J. H. Michael \cite{Michael} for linear stationary PDEs and extended to linear parabolic equations in \cite{Hoek}. We do not pursue this direction here, but just recall that the application of the classical method of continuity requires estimates up to the boundary along with a smooth operator, namely $F$ must be at least of class $C^1$. Moreover, we emphasize that the process of smoothing fully nonlinear operators is known only for concave/convex equations, and hence one needs a different approach for other classes of non-smooth operators.\\

\textit{Plan of the paper.} 
Section \ref{sec;survey} provides a brief survey of the known literature of the $C^{2,\alpha}$ solvability of fully nonlinear parabolic equations. Section \ref{sec;Isaacs} is devoted to the application of the maximum principle approach to obtain $C^{2+\alpha,1+\alpha/2}$ regularity for a time-dependent counterpart of an Isaacs model introduced by X. Cabr\'e and L. Caffarelli. We finally state an abstract result that provides $C^{2+\alpha,1+\alpha/2}$ regularity provided that $C^{2,\alpha}$ estimates for the elliptic problem are available.
\par\bigskip
\paragraph{\textit{Notation}} 
We denote by $B_r(x)$ the ball of center $x$ and radius $r$. When $x_0=0$ we simply write $B_r$.\\
$Q_r(x,t)$ is the parabolic cylinder $B_r(x)\times(t-r^2,t)$, and we write $Q_r$ when $(x,t)=(0,0)$.\\
We denote, given an open set $\Omega$, by $\partial_{\mathrm{par}}(\Omega\times(a,b))$ the parabolic boundary of a set.\\
We denote by $|u|_{0;\Omega}$ the sup-norm of $u$, i.e. $\|u\|_{C(\Omega)}$ (both in the elliptic and the parabolic case).\\
Let $Q\subset\Omega\times(0,T)$ and $\alpha\in(0,1)$. We denote by $d((x,t),(y,s))=|x-y|+|t-s|^\frac12$ the standard parabolic distance and
\begin{itemize}
\item $C^{\alpha,\alpha/2}(Q)$, $\alpha\in(0,1]$, the space of those $u:Q\to\R$ such that
\[
\|u\|_{C^{\alpha,\alpha/2}(Q)}:=|u|_{0;Q}+[u]_{C^{\alpha,\alpha/2}(Q)}=|u|_{0;Q}+\sup_{\substack{(x,t),(y,s)\in Q,\\ (x,t)\neq (y,s)}}\frac{|u(x,t)-u(y,s)|}{d^\alpha((x,t),(y,s))};
\]
\item $C^{1+\alpha,(1+\alpha)/2}(Q)$ the space of functions $u$ whose spatial gradient exists classically and equipped with the norm
\[
\|u\|_{C^{1+\alpha,(1+\alpha)/2}(Q)}:=|u|_{0;Q}+|Du|_{0;Q}+\sup_{\substack{(x,t),(y,s)\in Q,\\ (x,t)\neq (y,s)}}\frac{|u(x,t)-u(y,s)|}{d^{1+\alpha}((x,t),(y,s))}.
\]
In particular, any  $u\in C^{1+\alpha,(1+\alpha)/2}(Q)$ is such that each component of $Du$ belongs to $C^{\alpha,\alpha/2}(Q)$ and $u$ is H\"older continuous with exponent $(1+\alpha)/2$ in the time variable;
\item $C^{2+\alpha,1+\alpha/2}(Q)$ the space of functions $u$ such that
\[
\|u\|_{C^{2+\alpha,1+\alpha/2}(Q)}:=|u|_{0;Q}+\sum_{i=1}^n\|\partial_{x_i}u\|_{C^{1+\alpha,(1+\alpha)/2}(Q)}+\|\partial_t u\|_{C^{\alpha,\alpha/2}(Q)}
\]
This is equivalent to say that $D^2u$ belongs to $C^{\alpha,\alpha/2}(Q)$ and $\partial_t u$ belongs to $C^{\alpha,\alpha/2}(Q)$: it is a consequence of Remark 8.8.7 in \cite{KrylovBookHolder}. In this case we can consider the space $C^{2+\alpha,1+\alpha/2}(Q)$ equipped with the norm
\[
\|u\|_{C^{2+\alpha,1+\alpha/2}(Q)}:=|u|_{0;Q}+|Du|_{0;Q}+\|D^2u\|_{C^{\alpha,\alpha/2}(Q)}+\|\partial_t u\|_{C^{\alpha,\alpha/2}(Q)}.
\]
 For more properties on these spaces we refer to \cite{KrylovBookHolder}. Furthermore, we denote by
\[
\|u\|_{C^{2,1}(Q)}:=\sum_{2i+j\leq 2}|\partial_t^iD_x^j u|_{0;Q}.
\]
We will also use the equivalence (see e.g. p. 120 of \cite{KrylovBookHolder}) between the H\"older seminorm $[u]_{C^{\alpha,\alpha/2}(Q_1)}$ previously defined and the seminorm
\begin{align*}
[u]'_{C^{\alpha,\beta}(Q_1)}&=\sup_{t\in(-1,0)}\sup_{\substack{x,y\in B_1,\\ x\neq y}}\frac{|u(x,t)-u(y,t)|}{|x-y|^\alpha}+\sup_{x\in B_1}\sup_{\substack{t,s\in(-1,0),\\ t\neq s}}\frac{|u(x,t)-u(x,s)|}{|t-s|^\beta}\\
&=\sup_{t\in(-1,0)}[u(\cdot,t)]_{C^\alpha(B_1)}+\sup_{x\in B_1}[u(x,\cdot)]_{C^\beta((-1,0))}.
\end{align*}
\end{itemize}

\section{A survey on the (higher) regularity theory of fully nonlinear parabolic equations}\label{sec;survey}
In this section we survey on regularity results at the level of H\"older spaces for fully nonlinear parabolic equations, focusing on those of the form
\begin{equation}\label{parsurvey}
F(D^2u)-\partial_tu=0\text{ in }B_1\times(-1,0]\subset \R^{n+1}.
\end{equation}
We will in particular concentrate on structural conditions, possibly involving the dimension, guaranteeing classical solutions. The interested reader can find a complete updated account in the book \cite{FRRO} for the elliptic theory. We do not aim at discussing here $W^{2,1}_q$ regularity of solutions and we refer to \cite{Crandall,KrylovBookNew,Wang2} for more details and to the next section for the main results on the low-regularity theory. \\

The simplest result for \eqref{parsurvey} says that solutions to uniformly parabolic equations in $D=2+1$ (here we denote by $D=n+1$ the dimension in space-time) are always classical without any other assumption on $F$, see \cite{Andrews,Gjmpa,SilvestreAnnali} and the references therein. Its elliptic analogue was proved by L. Nirenberg (in the elliptic dimension $n=2$) \cite[Theorem 4.9]{FRRO}, and its proof, based on the De Giorgi-Nash-Moser theory applied to the equation solved by second derivatives, does not seem to generalize to parabolic equations. It is also important to mention a result of S. Kruzhkov showing that the same result holds in the $D=2$ case ($n=1$ variable in space), cf. \cite{K1} and Section XIV.7 in \cite{Lieberman}.
\begin{thm}\label{3D}
Assume that $F$ is uniformly elliptic and $n=2$. Then viscosity solutions to \eqref{parsurvey} (in dimension $n+1=3$) are always classical and belong to $C^{2+\alpha,1+\alpha/2}_{\mathrm{loc}}$ for some small universal $\alpha\in(0,1)$. It holds
\[
\|u\|_{C^{2+\alpha,1+\alpha/2}(Q_\frac12)}\leq C\|u\|_{L^\infty(Q_1)}.
\]
\end{thm}
A second result holds under Cordes-type assumptions on the ellipticity in any space dimension, see Theorem  6.4 in \cite{Gjmpa} or \cite{Vitor}.
\begin{thm}\label{Cordes}
Assume that $F$ is uniformly elliptic and $\frac{\Lambda}{\lambda}\leq 1+\delta$ for a small (explicit) constant $\delta$ depending only on $n$. Then viscosity solutions to \eqref{parsurvey} are always classical (in any dimension) and belong to $C^{2+\alpha,1+\alpha/2}_{\mathrm{loc}}$ for some small universal $\alpha$. In addition, we have the regularity estimate
\[
\|u\|_{C^{2+\alpha,1+\alpha/2}(Q_\frac12)}\leq C\|u\|_{L^\infty(Q_1)}.
\]
\end{thm}
We also mention that $C^{2+\alpha,1+\alpha/2}$ estimates can be obtained without any concavity assumption for short time horizons, see e.g. Chapter 8 of \cite{LunardiBook}, \cite{LionsSouganidis} and the references therein. More recently, maximal $C^{\alpha}$ estimates were also achieved for flat solutions (i.e. with small $L^\infty$ norm) in \cite{WangIUMJ} without concavity assumptions on $F$, extending a result due to O. Savin for stationary equations. These latter and the previous results are true for any uniformly parabolic equation, and hence for uniformly parabolic Isaacs equations given by
\[
\inf_{\eta\in A}\sup_{\gamma \in B} \{ L_{\eta\gamma} u - f_{\eta\gamma} \}=0,
\]
under appropriate restrictions (e.g. on the dimension, the coefficients, the solution or the time horizon).
We also emphasize that some counterexamples to the smoothness of solutions when $F$ is only uniformly parabolic can be found in the recent paper \cite{SilvestreAnnali}, which provides a time-dependent counterpart of the analysis in \cite{NTV}.\\

A counterpart of the result by Evans-Krylov \cite{Evans82,KrylovBookNew} provides parabolic $C^{2,\alpha}$ estimates under concavity assumptions on $F$, cf. \cite{Wang2}, see also \cite{Lieberman,KrylovBookOld,KrylovBookNew} for smooth operators.
\begin{thm}\label{EKpar}
Assume that $F$ is uniformly elliptic and concave. Then viscosity solutions to \eqref{parsurvey} are always classical and belong to $C^{2+\alpha,1+\alpha/2}_{\mathrm{loc}}$ for some small universal $\alpha$. It holds
\[
\|u\|_{C^{2+\alpha,1+\alpha/2}(Q_\frac12)}\leq C\|u\|_{L^\infty(Q_1)}.
\]
\end{thm}
Concavity on the second derivatives was recently weakened to the requirement of convexity of the superlevel sets, cf. \cite{Gjmpa}. The latter also contains some $C^{1,1}$ estimates under concavity-type conditions at infinity implying, among others, $W^{2,1}_q$ estimates for non-homogeneous equations. A recent account on the treatment of fully nonlinear parabolic equations with relaxed concavity conditions can be found in \cite{KrylovBookNew}.\\
We also mention that slightly before the Evans-Krylov theory, the paper \cite{Schoenauer} provided an obstacle problem approach to reach parabolic $C^{2,\beta}$, $\beta\in(0,1)$ small and universal, regularity for the model two-operator equation
\[
\max\{\partial_t u-L_1u,\partial_tu-L_2u\}=0
\]
(here $L_1,L_2$ are uniformly elliptic linear operators) by reducing the problem to a variational inequality, an idea introduced by Br\'ezis-Evans \cite{BrezisEvans} for stationary equations. More references on the early regularity theory of fully nonlinear elliptic and parabolic equations can be found in \cite[p. 382-384]{Lieberman}. We are not aware of any nonconcave/nonconvex $F$ guaranteeing $C^{2+\alpha,1+\alpha/2}$ estimates in the time-dependent framework, except for \cite{StreetsWarren} and the aforementioned papers dealing with certain restrictions on the data of the problem.

\section{$C^{2+\alpha,1+\alpha/2}$ regularity for fully nonlinear parabolic Isaacs equations}\label{sec;Isaacs}
Our main result is an Evans-Krylov theorem for
\begin{equation}\label{parIsaacs}
F(D^2u)-\partial_t u=0\text{ in }Q_1=B_1\times(-1,0].
\end{equation}
with
\begin{equation}\label{Fcc}
\begin{cases}
F(M)=\min\{F^\cap (M),F^\cup(M)\},\ \forall M\in\mathcal{S}_n\\
F(0)=0,\ F^\cap,F^\cup\text{ are uniformly elliptic}\\
F^\cap\text{ is concave in $M$},\ F^\cup\text{ is convex in $M$}.
\end{cases}
\end{equation}
We recall that the Evans-Krylov theorem allows us to pass from $C^{1,1}$ to $C^{2,\alpha}$ regularity in the stationary case, see e.g. \cite[Chapter 6]{CC}.
\begin{thm}\label{mainpar1}
Let $u\in C(Q_1)$ be a viscosity solution to \eqref{parIsaacs} with $F$ satisfying \eqref{Fcc}. Then for some universal $\tilde{\alpha}\in(0,1)$ depending on $n,\lambda,\Lambda$ we have
\[
\|u\|_{C^{2+\tilde\alpha,1+\tilde\alpha/2}(\overline{Q}_\frac12)}\leq C\|u\|_{L^\infty(Q_1)}.
\]
where $C$ is a constant depending on $\lambda,\Lambda,n,\tilde\alpha$.
\end{thm}
\begin{rem}
Differently from the elliptic case, the condition $F(0)=0$ cannot be in general dropped, as outlined in \cite[Remark p. 257]{KrylovBookNew}.
\end{rem}
As a corollary, an argument by L. Caffarelli \cite[Section 8.1]{CC}, see also \cite[Theorem 1.1]{Wang2} for the time-dependent case, provides the maximal $C^\alpha$ regularity, $\alpha<\tilde\alpha$, for the non-homogeneous equation
\[
F(x,t,D^2u)-\partial_t u=f(x,t)
\]
under the following assumptions on $F$:
\begin{itemize}
\item[(i)] For every $(x_0,t_0)\in Q_1$, the operator $F(x_0,t_0,\cdot)$ is the minimum of a concave and a convex operator (possibly depending on $(x_0,t_0)$);
\item[(ii)] $F(\cdot,M)$ and $f(\cdot)$ are H\"older continuous functions with exponent $\alpha$ with respect to the parabolic distance.
\end{itemize}
This implies that $F(x,t,D^2u)$ can be an Isaacs operator of the form \eqref{IsaacsIntro} with variable H\"older continuous coefficients with respect to the parabolic distance.\\

We denote by $\mathcal{M}^\pm_{\lambda,\Lambda}$ the Pucci's extremal operators with constants $0<\lambda\leq \Lambda$ defined by
\[
\mathcal{M}^+_{\lambda,\Lambda}(M)=\sup_{\lambda I_n\leq A\leq \Lambda I_n}\mathrm{Tr}(AM)=\Lambda\sum_{e_k>0}e_k+\lambda\sum_{e_k<0}e_k,
\]
\[
\mathcal{M}^-_{\lambda,\Lambda}(M)=\sup_{\lambda I_n\leq A\leq \Lambda I_n}\mathrm{Tr}(AM)=\lambda\sum_{e_k>0}e_k+\Lambda\sum_{e_k<0}e_k,
\]
 where $e_k=e_k(M)$ are the eigenvalues of $M$. Moreover, $\underline{\mathcal{S}}(\lambda,\Lambda,f)$ stands for the space of continuous functions $u$ in $\Omega$ that are subsolutions of the evolutive Pucci's maximal equation
\[
\mathcal{M}^+_{\lambda,\Lambda}(D^2u)-\partial_t u= f(x,t)\text{ in the viscosity sense in $\Omega$},
\]
Similarly, $\overline{\mathcal{S}}(\lambda,\Lambda,f)$ will denote the space of continuous functions $u$ in $\Omega$ that are supersolutions to the Pucci's minimal equation
\[
\mathcal{M}^-_{\lambda,\Lambda}(D^2u)-\partial_tu= f(x,t)\text{ in the viscosity sense in $\Omega$}.
\]
Solutions of fully nonlinear uniformly elliptic equations $F(D^2u)-\partial_tu=0$ belong to the class $\mathcal{S}(\lambda,\Lambda,0)=\overline{\mathcal{S}}(\lambda,\Lambda,0)\cap \underline{\mathcal{S}}(\lambda,\Lambda,0)$.\\

Before proving Theorem \ref{mainpar1} we recall some standard facts, see \cite{KrylovSafonov} and \cite{ImbertSilvestre}:
\begin{thm}[Krylov-Safonov parabolic H\"older regularity]\label{KS}
Let $u\in C(Q_1)$ be a solution of the viscosity inequalities 
\[
\begin{cases}
\mathcal{M}^+_{\lambda,\Lambda}(D^2u)-\partial_t u\geq0\text{ in }Q_1,\\
\mathcal{M}^-_{\lambda,\Lambda}(D^2u)-\partial_t u\leq0\text{ in }Q_1.
\end{cases}
\]
 Then, for some $\alpha>0$ depending on $n,\lambda,\Lambda$ we have that $u\in C^{\alpha,\alpha/2}(Q_\frac12)$ and the following regularity estimate holds
\[
\|u\|_{C^{\alpha,\alpha/2}(Q_\frac12)}\leq C\|u\|_{L^\infty(Q_1)},
\]
where $C$ is a constant depending on $\lambda,\Lambda,n$.
\end{thm}
This H\"older property is the best one can get for equations with bounded and measurable coefficients in the parabolic dimension $D=n+1\geq 2$ by \cite[Theorem 2.2 and the discussion below]{SilvestreAnnali}. The previous estimates are the cornerstone to deduce higher regularity of first-order: the next result holds for viscosity solutions of any fully nonlinear parabolic homogeneous equations without any other assumption on $F$ other than the uniform ellipticity, cf. Theorem 2.3 in \cite{SilvestreAnnali}, Theorems 4.8 and 4.9 in \cite{Wang2} or Lemma 3 p. 257 in \cite{KrylovBookNew} (see also Corollary 5.7 in \cite{CC} and \cite[Theorem 4.24]{FRRO} for the elliptic case).
\begin{thm}[First-order space-time H\"older regularity]\label{firstorder}
Let $u\in C(Q_1)$ be a viscosity solution of $F(D^2u)-\partial_tu=0$, $F$ uniformly elliptic. Then, for some $\theta>0$ depending on $n,\lambda,\Lambda$ we have that $\partial_tu,Du\in C^{\theta,\theta/2}(Q_\frac12)$ and 
\[
\|\partial_tu\|_{C^{\theta,\theta/2}(Q_\frac12)}+\|Du\|_{C^{\theta,\theta/2}(Q_\frac12)}\leq C\|u\|_{L^\infty(Q_1)},
\]
where $C$ depends on $\lambda,\Lambda,n$.
\end{thm}
We now state the main result we need to prove parabolic maximal $C^\alpha$ regularity for \eqref{parIsaacs}. This is an interior $C^{2,\beta}$ estimate, $\beta$ small and universal, for the stationary non-homogeneous equation driven by operators satisfying \eqref{Fcc}, see Corollary 1.3 and Remark 1.4 of \cite{CCjmpa}.
\begin{thm}\label{EKell}
Let $u\in C(B_1)$ be a viscosity solution to $F(D^2u)=f(x)$ in $B_1$ with $f\in C^\gamma(B_1)$ for some $0<\gamma<\bar\gamma$, $\bar\gamma\in(0,1)$ being a universal constant, and $F$ satisfying \eqref{Fcc}. Then, the following estimate holds
\[
\|u\|_{C^{2,\gamma}(\overline{B}_\frac12)}\leq C(\|u\|_{L^\infty(B_1)}+\|f\|_{C^{\gamma}(\overline B_\frac34)}),
\]
where $C$ depends on $n,\lambda,\Lambda,\gamma$.
\end{thm}
\begin{rem}
The universal exponent $\bar\gamma\in(0,1)$ is the H\"older regularity exponent of second derivatives of solutions solving the constant coefficient equation, where $F$ is uniformly elliptic and satisfies \eqref{Fcc}.
\end{rem}
Next we exploit the above $C^{2,\gamma}_x$ estimate and the evolution equation to improve the time-H\"older exponent of first-order derivatives from Theorem \ref{firstorder} via a comparison principle argument. This is inspired by \cite{Andrews}.
\begin{prop}\label{improve}
Suppose that there exist $\beta\in(0,1)$ and a constant $C>0$ such that \[|D^2u(x,\cdot)-D^2u(y,\cdot)|\leq C|x-y|^{\beta},\ x,y\in B_\frac12.\] Any $C^{2,1}$ solution to the uniformly parabolic equation \eqref{parIsaacs} (without any other structural assumption on $F$) satisfies the estimate
\[
|Du(\cdot,t)-Du(\cdot,s)|\leq C'|t-s|^{\frac{1+\beta}{2}}\text{ in }B_\frac{1}{16}.
\]
where $C'$ is a positive constant depending on $C, n,\beta, \lambda,\Lambda$.
\end{prop}

\begin{rem}
Note that the validity of the H\"older regularity condition on the second derivatives hides some additional structural conditions on $F$ beyond the uniform ellipticity, unless the space dimension $n=1,2$.
\end{rem}

\begin{proof}
First note that for $e\in\R^n$, $|e|=1$,
\[
u_h(x,t):=\frac{u(x+he,t)-u(x,t)}{h}\in\mathcal{S}(\lambda/n,\Lambda),
\]
i.e. $u_h$ is a solution to a nondivergence form equation in a smaller cylinder. This means (see e.g. \cite[Proposition 5.5]{CC} or \cite[Theorem 4.6]{Wang2}) that we have the validity of the following inequalities
\begin{equation}\label{eq1}
\mathcal{M}^+_{\lambda,\Lambda}(D^2u_h)-\partial_t u_h\geq0\text{ in }Q_{\frac12-h}
\end{equation}
and
\begin{equation}\label{eq2}
\mathcal{M}^-_{\lambda,\Lambda}(D^2u_h)-\partial_t u_h\leq0\text{ in }Q_{\frac12-h}.
\end{equation}
Moreover, the assumptions on the regularity of $D^2u$ in space (for fixed times) imply that $u_h\in C^{1,\beta}_x$, $\beta\in(0,1)$, uniformly with respect to $h$ .\\
We give a proof using the maximum principle and the equivalence of the H\"older seminorms with those given by Taylor polynomials \cite[Sections 3.3 and 8.5]{KrylovBookHolder}. We have by Theorem 3.3.1 in \cite{KrylovBookHolder}
\[
|u_h(z',t)-u_h(z,t)-Du_h(z,t)\cdot (z'-z)|\leq C|z'-z|^{1+\beta}\text{ on }Q_{\frac12-h}.
\]
We apply the weighted Young's inequality to find for a positive constant $\widetilde{C}$
\[
C|z'-z|^{1+\beta}\leq \eps+\widetilde C\eps^{-\frac{1-\beta}{1+\beta}}|z'-z|^2.
\]
This implies the inequality
\[
u_h(z',t)\leq u_h(z,t)+Du_h(z,t)\cdot (z'-z)+\eps+\widetilde C\eps^{-\frac{1-\beta}{1+\beta}}|z'-z|^2.
\]
We define
\[
\Phi^+(z',t')=u_h(z,t)+Du_h(z,t)\cdot (z'-z)+\eps+\widetilde C\eps^{-\frac{1-\beta}{1+\beta}}|z'-z|^2+\nu(t'-t).
\]
One can check, after choosing $\nu\geq2\widetilde Cn\Lambda\eps^{-\frac{1-\beta}{1+\beta}}$, the validity of
\[
\mathcal{M}^+_{\lambda,\Lambda}(D^2\Phi^+)-\partial_t\Phi^+\leq0, 
\]
so $\Phi^+$ is a classical supersolution to the equation solved by $u_h$ in $Q_\frac{1}{4}$, provided that $h<\frac{1}{8}$. Recall that $u_h\in\mathcal{S}(\lambda/n,\Lambda)$ and thus solves a linear parabolic homogeneous equation in nondivergence form, see \eqref{eq1}-\eqref{eq2}. Choosing $\eps$ small enough, $\eps<K\|u\|_{C^2}^{-\frac{1-\beta}{1+\beta}}$ for a suitable $K>0$, we have that $\Phi^+\geq u_h$ on $\partial B_\frac{1}{4}\times(-\frac{1}{16},0]$. By the maximum principle \cite[Theorem 8.2]{CIL} $\Phi^+\geq u_h$ in the interior of the cylinder. Taking $z'=z$ and optimizing with respect to $\eps$ we conclude
\[
u_h(z,t')\leq u_h(z,t)+\bar C_1(t'-t)^\frac{1+\beta}{2},
\]
for a positive constant $\bar C_1$. We can get a similar estimate from below by constructing a subsolution, say $\Phi^-$, to the PDE
\[
\mathcal{M}^-_{\lambda,\Lambda}(D^2\Phi^-)-\partial_t\Phi^-\geq0
\]
and applying the comparison principle. More precisely, we have
\[
u_h(z',t)\geq u_h(z,t)+Du_h(z,t)\cdot (z'-z)-\eps-\widetilde C\eps^{-\frac{1-\beta}{1+\beta}}|z'-z|^2.
\]
We define
\[
\Phi^{-}(z',t')=u_h(z,t)+Du_h(z,t)\cdot (z'-z)-\eps-\widetilde C\eps^{-\frac{1-\beta}{1+\beta}}|z'-z|^2-\nu(t'-t)
\]
and choose $\nu\geq2\widetilde Cn\lambda\eps^{-\frac{1-\beta}{1+\beta}}$ to conclude that $\mathcal{M}^-_{\lambda,\Lambda}(D^2\Phi^-)-\partial_t\Phi^-\geq0$. By the maximum principle we have $\Phi^-\leq u_h$ in the interior of the cylinder, and proceeding as above we have for a positive constant $\bar C_2$ the estimate
\[
u_h(z,t')\geq u_h(z,t)-\bar C_2(t'-t)^\frac{1+\beta}{2}.
\]
\end{proof}
We are now ready with the proof of the main result:
\begin{proof}[Proof of Theorem \ref{mainpar1}]
Theorem \ref{firstorder} shows that viscosity solutions of \eqref{parIsaacs} satisfy $\partial_t u\in C^{\theta,\theta/2}$ for some small $\theta\in(0,1)$ depending on $\lambda,\Lambda$ and $n$ ($\theta$ is the exponent of Theorem \ref{firstorder}). Recall that such a result does not depend on the structure of the operator and it requires only  its uniform ellipticity. Therefore, we can freeze the time-variable and consider the solution $u=u(\cdot,t)$ of
\[
F(D^2u(\cdot,t))=\partial_t u(\cdot,t)\in C^{\theta}.
\]
Note that $\partial_t u(\cdot,t)$ is H\"older continuous with a possibly smaller exponent by the inclusion of H\"older spaces, namely $\partial_t u(\cdot,t)\in C^{\tilde\alpha}$, $\tilde\alpha\in(0,\bar\gamma)\cap (0,\theta]$. By Theorem \ref{EKell} we have that $D^2u(\cdot,t)$ exists and it is H\"older continuous with an exponent $\tilde\alpha\in(0,\bar\gamma)\cap (0,\theta]$ on each time slice, namely
\[
|D^2u(x,\cdot)-D^2u(y,\cdot)|\leq C|x-y|^{\tilde\alpha}.
\]
Applying Proposition \ref{improve} with $\beta=\tilde\alpha$ we conclude that
\[
|Du(\cdot,t)-Du(\cdot,s)|\leq C|t-s|^{\frac{1+\tilde\alpha}{2}}.
\]
Therefore, $D^2u$ is time-H\"older continuous with exponent $\tilde\alpha/2$ after invoking the interpolation argument from Lemma 3.1 p. 80 of \cite{LSU} applied to $Du$, see also \cite[Section 3.4]{Andrews}. We recall briefly that it says that if a function $v$ is $C^{\gamma_1}_t$ and $Dv\in C_x^{\gamma_2}$, $\gamma_1,\gamma_2\in(0,1)$, then $Dv\in C^{\gamma_1\gamma_2/(1+\gamma_2)}_t$. This implies that $D^2u\in C^{\tilde \alpha,\tilde \alpha/2}$. \\

\end{proof}

\begin{rem}\label{compare}
It is worth comparing the strategy outlined above with those already appeared in the literature of fully nonlinear parabolic equations. There are essentially two ways to achieve $C^{2+\alpha,1+\alpha/2}$ estimates for homogeneous fully nonlinear parabolic equations. The first one is to prove Evans-Krylov estimates directly for parabolic equations, as it is done in \cite{KrylovBookOld} for concave equations. The second one is to consider \eqref{parIsaacs} as $F(D^2u)=\partial_t u$, give an estimate of $\partial_t u$ and then consider the PDE for any fixed $t$ as an elliptic one: this viewpoint however requires the development of a $C^{2,\alpha}$ theory for the non-homogeneous elliptic equation $F(D^2u)=f(x)\in C^\alpha$. The approach carried out in Chapter 12 of \cite{KrylovBookNew}, see also \cite[Appendix A]{BrendleHuisken}, follows this second idea, and exploits an interpolation argument between $D^2u(\cdot,t)\in C^\alpha_x$ and $\partial_t u(\cdot,t)\in C^\alpha_x$ for fixed times $t$ (cf. Lemma 4 p. 259 in \cite{KrylovBookNew}), a slightly different way with respect to our method. Related ideas were also employed in \cite{Andrews,SilvestreAnnali} for parabolic equations in 3D. 
\end{rem}

The foregoing approach applies to deduce interior $C^{2+\alpha,1+\alpha/2}$ bounds in the case of concave operators with a different proof than Theorem 4.13 in \cite{Wang2} (in particular without using Lemma 4.16 therein), \cite[Theorem 12.2.1]{KrylovBookNew} and \cite[Appendix A]{BrendleHuisken}. Here we provided a general strategy to deduce parabolic $C^{2+\alpha,1+\alpha/2}$ regularity through the maximum principle and the $C^{2,\alpha}$ regularity in the space variable only. We can generalize the previous argument providing an abstract result under the following hypothesis:
\begin{itemize}
\item[(H)] The stationary non-homogeneous equation $F(D^2u)=f$, $f\in C^\alpha$, has interior $C^{2,\alpha}$ estimates for some universal $\alpha\in(0,1)$ with constant $c_e$.
\end{itemize}
\begin{thm}
Let $u\in C(Q_1)$ be a viscosity solution of the equation $F(D^2u)-\partial_tu=0$ in $Q_1$. Assume that $F$ is uniformly elliptic, $F(0)=0$, and that (H) holds (i.e. the nonhomogeneous equation admits local $C^{2,\alpha}$ estimates). Then, $u\in C^{2+\alpha,1+\alpha/2}(Q_\frac12)$ and the following estimate holds
\[
\|u\|_{C^{2+\alpha,1+\alpha/2}(Q_\frac12)}\leq C
\] 
for a constant $C$ depending on $n,\lambda,\Lambda,\alpha,c_e$.
\end{thm}
\begin{proof}
Since $F$ is uniformly elliptic, we have $\partial_t u,Du\in C^{\alpha,\alpha/2}$. Then, the assumption (H) yields $D^2u\in C^{\alpha}_x$, uniformly in time by considering the equation $F(D^2u(\cdot,t))=\partial_t u(\cdot.t)$. Therefore, we can proceed as in the proof of Theorem \ref{mainpar1} to prove that $D^2u\in C_t^{\alpha/2}$.
\end{proof}


\end{document}